\documentclass[a4paper,10pt]{article}
\usepackage[utf8]{inputenc}
\usepackage[english]{babel}
\usepackage{hyperref}

\usepackage{amsmath}
\usepackage{amsthm}
\usepackage{amssymb}
\usepackage{amsfonts}
\usepackage[alphabetic,abbrev]{amsrefs}

\usepackage{import}

\usepackage{tikz}
\usetikzlibrary{matrix, arrows}

\usepackage{calc}

\newtheorem{thm}{Theorem}[section]
\newtheorem{prop}[thm]{Proposition}
\newtheorem{lemma}[thm]{Lemma}
\newtheorem{cor}[thm]{Corollary}
\newtheorem{defn}[thm]{Definition}
\newtheorem{ex}[thm]{Example}

\newtheorem{rem}[thm]{Remark}

{%
\vskip 6pt%
\noindent%
\textbf{Definition #1. }%
\ignorespaces%
}%
{%
\par\noindent\ignorespacesafterend%
}

\newenvironment{myThm}[1]%
{%
\vskip 6pt%
\noindent%
\textbf{Theorem #1. }%
\ignorespaces%
}%
{%
\par\noindent\ignorespacesafterend%
}

{%
\vskip 6pt%
\noindent%
\textbf{Proposition #1. }%
\ignorespaces%
}%
{%
\par\noindent\ignorespacesafterend%
}

{%
\vskip 6pt%
\noindent%
\textbf{Corollary #1. }%
\ignorespaces%
}%
{%
\par\noindent\ignorespacesafterend%
}

\newcommand{\z}[1]{\mathbb{Z}/{#1}}
\newcommand{\zzz}{\mathbb{Z}}
\newcommand{\qqq}{\mathbb{Q}}

\newcommand{\qdim}{\operatorname{qdim}}

\newcommand{\supp}{\operatorname{supp}}

\newcommand{\GpAlg}[2]{\mathbb{#1}\left[ #2 \right]}

\newcommand{\gpAlg}[1]{\GpAlg{Q}{#1}}

\newcommand{\gpRing}[1]{\GpAlg{Z}{#1}}
\newcommand{\cyclicGroupAlg}[1]{\gpAlg{\z{#1}}}

\newcommand{\cyclotomicField}[1]{\gpAlg{\xi_{#1}}}
\newcommand{\Iso}{\operatorname{Iso}}

\newcommand{\Hom}[4]{\operatorname{Hom}^{#1}_{#2} \left(#3, #4 \right)}

\newcommand{\cross}{\operatorname{Cr}}
\newcommand{\nonorientCross}{\raisebox{-5pt}{
\begingroup%
  \makeatletter%
  \providecommand\color[2][]{%
    \errmessage{(Inkscape) Color is used for the text in Inkscape, but the package 'color.sty' is not loaded}%
    \renewcommand\color[2][]{}%
  }%
  \providecommand\transparent[1]{%
    \errmessage{(Inkscape) Transparency is used (non-zero) for the text in Inkscape, but the package 'transparent.sty' is not loaded}%
    \renewcommand\transparent[1]{}%
  }%
  \providecommand\rotatebox[2]{#2}%
  \ifx\svgwidth\undefined%
    \setlength{\unitlength}{18.70443764bp}%
    \ifx\svgscale\undefined%
      \relax%
    \else%
      \setlength{\unitlength}{\unitlength * \real{\svgscale}}%
    \fi%
  \else%
    \setlength{\unitlength}{\svgwidth}%
  \fi%
  \global\let\svgwidth\undefined%
  \global\let\svgscale\undefined%
  \makeatother%
  \begin{picture}(1,0.91836368)%
    \put(0,0){\includegraphics[width=\unitlength,page=1]{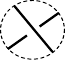}}%
  \end{picture}%
\endgroup%
}}

\newcommand{\posCross}{\raisebox{-5pt}{
\begingroup%
  \makeatletter%
  \providecommand\color[2][]{%
    \errmessage{(Inkscape) Color is used for the text in Inkscape, but the package 'color.sty' is not loaded}%
    \renewcommand\color[2][]{}%
  }%
  \providecommand\transparent[1]{%
    \errmessage{(Inkscape) Transparency is used (non-zero) for the text in Inkscape, but the package 'transparent.sty' is not loaded}%
    \renewcommand\transparent[1]{}%
  }%
  \providecommand\rotatebox[2]{#2}%
  \ifx\svgwidth\undefined%
    \setlength{\unitlength}{18.70443764bp}%
    \ifx\svgscale\undefined%
      \relax%
    \else%
      \setlength{\unitlength}{\unitlength * \real{\svgscale}}%
    \fi%
  \else%
    \setlength{\unitlength}{\svgwidth}%
  \fi%
  \global\let\svgwidth\undefined%
  \global\let\svgscale\undefined%
  \makeatother%
  \begin{picture}(1,0.91836368)%
    \put(0,0){\includegraphics[width=\unitlength,page=1]{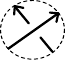}}%
  \end{picture}%
\endgroup%
}}
\newcommand{\negCross}{\raisebox{-5pt}{
\begingroup%
  \makeatletter%
  \providecommand\color[2][]{%
    \errmessage{(Inkscape) Color is used for the text in Inkscape, but the package 'color.sty' is not loaded}%
    \renewcommand\color[2][]{}%
  }%
  \providecommand\transparent[1]{%
    \errmessage{(Inkscape) Transparency is used (non-zero) for the text in Inkscape, but the package 'transparent.sty' is not loaded}%
    \renewcommand\transparent[1]{}%
  }%
  \providecommand\rotatebox[2]{#2}%
  \ifx\svgwidth\undefined%
    \setlength{\unitlength}{18.70443764bp}%
    \ifx\svgscale\undefined%
      \relax%
    \else%
      \setlength{\unitlength}{\unitlength * \real{\svgscale}}%
    \fi%
  \else%
    \setlength{\unitlength}{\svgwidth}%
  \fi%
  \global\let\svgwidth\undefined%
  \global\let\svgscale\undefined%
  \makeatother%
  \begin{picture}(1,0.91836368)%
    \put(0,0){\includegraphics[width=\unitlength,page=1]{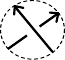}}%
  \end{picture}%
\endgroup%
}}
\newcommand{\orientRes}{\raisebox{-5pt}{
\begingroup%
  \makeatletter%
  \providecommand\color[2][]{%
    \errmessage{(Inkscape) Color is used for the text in Inkscape, but the package 'color.sty' is not loaded}%
    \renewcommand\color[2][]{}%
  }%
  \providecommand\transparent[1]{%
    \errmessage{(Inkscape) Transparency is used (non-zero) for the text in Inkscape, but the package 'transparent.sty' is not loaded}%
    \renewcommand\transparent[1]{}%
  }%
  \providecommand\rotatebox[2]{#2}%
  \ifx\svgwidth\undefined%
    \setlength{\unitlength}{18.70443764bp}%
    \ifx\svgscale\undefined%
      \relax%
    \else%
      \setlength{\unitlength}{\unitlength * \real{\svgscale}}%
    \fi%
  \else%
    \setlength{\unitlength}{\svgwidth}%
  \fi%
  \global\let\svgwidth\undefined%
  \global\let\svgscale\undefined%
  \makeatother%
  \begin{picture}(1,0.91836368)%
    \put(0,0){\includegraphics[width=\unitlength,page=1]{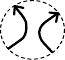}}%
  \end{picture}%
\endgroup%
}}
\newcommand{\nonorientRes}{
\begingroup%
  \makeatletter%
  \providecommand\color[2][]{%
    \errmessage{(Inkscape) Color is used for the text in Inkscape, but the package 'color.sty' is not loaded}%
    \renewcommand\color[2][]{}%
  }%
  \providecommand\transparent[1]{%
    \errmessage{(Inkscape) Transparency is used (non-zero) for the text in Inkscape, but the package 'transparent.sty' is not loaded}%
    \renewcommand\transparent[1]{}%
  }%
  \providecommand\rotatebox[2]{#2}%
  \ifx\svgwidth\undefined%
    \setlength{\unitlength}{18.70443764bp}%
    \ifx\svgscale\undefined%
      \relax%
    \else%
      \setlength{\unitlength}{\unitlength * \real{\svgscale}}%
    \fi%
  \else%
    \setlength{\unitlength}{\svgwidth}%
  \fi%
  \global\let\svgwidth\undefined%
  \global\let\svgscale\undefined%
  \makeatother%
  \begin{picture}(1,0.91836368)%
    \put(0,0){\includegraphics[width=\unitlength,page=1]{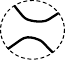}}%
  \end{picture}%
\endgroup%
}
\newcommand{\nonorientOrbit}{\raisebox{-5pt}{
\begingroup%
  \makeatletter%
  \providecommand\color[2][]{%
    \errmessage{(Inkscape) Color is used for the text in Inkscape, but the package 'color.sty' is not loaded}%
    \renewcommand\color[2][]{}%
  }%
  \providecommand\transparent[1]{%
    \errmessage{(Inkscape) Transparency is used (non-zero) for the text in Inkscape, but the package 'transparent.sty' is not loaded}%
    \renewcommand\transparent[1]{}%
  }%
  \providecommand\rotatebox[2]{#2}%
  \ifx\svgwidth\undefined%
    \setlength{\unitlength}{18.70443764bp}%
    \ifx\svgscale\undefined%
      \relax%
    \else%
      \setlength{\unitlength}{\unitlength * \real{\svgscale}}%
    \fi%
  \else%
    \setlength{\unitlength}{\svgwidth}%
  \fi%
  \global\let\svgwidth\undefined%
  \global\let\svgscale\undefined%
  \makeatother%
  \begin{picture}(1,0.91836368)%
    \put(0,0){\includegraphics[width=\unitlength,page=1]{crossing_nonoriented.pdf}}%
  \end{picture}%
\endgroup%
} \ldots \raisebox{-5pt}{
\begingroup%
  \makeatletter%
  \providecommand\color[2][]{%
    \errmessage{(Inkscape) Color is used for the text in Inkscape, but the package 'color.sty' is not loaded}%
    \renewcommand\color[2][]{}%
  }%
  \providecommand\transparent[1]{%
    \errmessage{(Inkscape) Transparency is used (non-zero) for the text in Inkscape, but the package 'transparent.sty' is not loaded}%
    \renewcommand\transparent[1]{}%
  }%
  \providecommand\rotatebox[2]{#2}%
  \ifx\svgwidth\undefined%
    \setlength{\unitlength}{18.70443764bp}%
    \ifx\svgscale\undefined%
      \relax%
    \else%
      \setlength{\unitlength}{\unitlength * \real{\svgscale}}%
    \fi%
  \else%
    \setlength{\unitlength}{\svgwidth}%
  \fi%
  \global\let\svgwidth\undefined%
  \global\let\svgscale\undefined%
  \makeatother%
  \begin{picture}(1,0.91836368)%
    \put(0,0){\includegraphics[width=\unitlength,page=1]{crossing_nonoriented.pdf}}%
  \end{picture}%
\endgroup%
}}
\newcommand{\posOrbit}{\raisebox{-5pt}{
\begingroup%
  \makeatletter%
  \providecommand\color[2][]{%
    \errmessage{(Inkscape) Color is used for the text in Inkscape, but the package 'color.sty' is not loaded}%
    \renewcommand\color[2][]{}%
  }%
  \providecommand\transparent[1]{%
    \errmessage{(Inkscape) Transparency is used (non-zero) for the text in Inkscape, but the package 'transparent.sty' is not loaded}%
    \renewcommand\transparent[1]{}%
  }%
  \providecommand\rotatebox[2]{#2}%
  \ifx\svgwidth\undefined%
    \setlength{\unitlength}{18.70443764bp}%
    \ifx\svgscale\undefined%
      \relax%
    \else%
      \setlength{\unitlength}{\unitlength * \real{\svgscale}}%
    \fi%
  \else%
    \setlength{\unitlength}{\svgwidth}%
  \fi%
  \global\let\svgwidth\undefined%
  \global\let\svgscale\undefined%
  \makeatother%
  \begin{picture}(1,0.91836368)%
    \put(0,0){\includegraphics[width=\unitlength,page=1]{crossing_pos.pdf}}%
  \end{picture}%
\endgroup%
} \ldots \raisebox{-5pt}{
\begingroup%
  \makeatletter%
  \providecommand\color[2][]{%
    \errmessage{(Inkscape) Color is used for the text in Inkscape, but the package 'color.sty' is not loaded}%
    \renewcommand\color[2][]{}%
  }%
  \providecommand\transparent[1]{%
    \errmessage{(Inkscape) Transparency is used (non-zero) for the text in Inkscape, but the package 'transparent.sty' is not loaded}%
    \renewcommand\transparent[1]{}%
  }%
  \providecommand\rotatebox[2]{#2}%
  \ifx\svgwidth\undefined%
    \setlength{\unitlength}{18.70443764bp}%
    \ifx\svgscale\undefined%
      \relax%
    \else%
      \setlength{\unitlength}{\unitlength * \real{\svgscale}}%
    \fi%
  \else%
    \setlength{\unitlength}{\svgwidth}%
  \fi%
  \global\let\svgwidth\undefined%
  \global\let\svgscale\undefined%
  \makeatother%
  \begin{picture}(1,0.91836368)%
    \put(0,0){\includegraphics[width=\unitlength,page=1]{crossing_pos.pdf}}%
  \end{picture}%
\endgroup%
}}
\newcommand{\negOrbit}{\raisebox{-5pt}{
\begingroup%
  \makeatletter%
  \providecommand\color[2][]{%
    \errmessage{(Inkscape) Color is used for the text in Inkscape, but the package 'color.sty' is not loaded}%
    \renewcommand\color[2][]{}%
  }%
  \providecommand\transparent[1]{%
    \errmessage{(Inkscape) Transparency is used (non-zero) for the text in Inkscape, but the package 'transparent.sty' is not loaded}%
    \renewcommand\transparent[1]{}%
  }%
  \providecommand\rotatebox[2]{#2}%
  \ifx\svgwidth\undefined%
    \setlength{\unitlength}{18.70443764bp}%
    \ifx\svgscale\undefined%
      \relax%
    \else%
      \setlength{\unitlength}{\unitlength * \real{\svgscale}}%
    \fi%
  \else%
    \setlength{\unitlength}{\svgwidth}%
  \fi%
  \global\let\svgwidth\undefined%
  \global\let\svgscale\undefined%
  \makeatother%
  \begin{picture}(1,0.91836368)%
    \put(0,0){\includegraphics[width=\unitlength,page=1]{crossing_neg.pdf}}%
  \end{picture}%
\endgroup%
} \ldots \raisebox{-5pt}{
\begingroup%
  \makeatletter%
  \providecommand\color[2][]{%
    \errmessage{(Inkscape) Color is used for the text in Inkscape, but the package 'color.sty' is not loaded}%
    \renewcommand\color[2][]{}%
  }%
  \providecommand\transparent[1]{%
    \errmessage{(Inkscape) Transparency is used (non-zero) for the text in Inkscape, but the package 'transparent.sty' is not loaded}%
    \renewcommand\transparent[1]{}%
  }%
  \providecommand\rotatebox[2]{#2}%
  \ifx\svgwidth\undefined%
    \setlength{\unitlength}{18.70443764bp}%
    \ifx\svgscale\undefined%
      \relax%
    \else%
      \setlength{\unitlength}{\unitlength * \real{\svgscale}}%
    \fi%
  \else%
    \setlength{\unitlength}{\svgwidth}%
  \fi%
  \global\let\svgwidth\undefined%
  \global\let\svgscale\undefined%
  \makeatother%
  \begin{picture}(1,0.91836368)%
    \put(0,0){\includegraphics[width=\unitlength,page=1]{crossing_neg.pdf}}%
  \end{picture}%
\endgroup%
}}
\newcommand{\orientResOrbit}{\raisebox{-5pt}{
\begingroup%
  \makeatletter%
  \providecommand\color[2][]{%
    \errmessage{(Inkscape) Color is used for the text in Inkscape, but the package 'color.sty' is not loaded}%
    \renewcommand\color[2][]{}%
  }%
  \providecommand\transparent[1]{%
    \errmessage{(Inkscape) Transparency is used (non-zero) for the text in Inkscape, but the package 'transparent.sty' is not loaded}%
    \renewcommand\transparent[1]{}%
  }%
  \providecommand\rotatebox[2]{#2}%
  \ifx\svgwidth\undefined%
    \setlength{\unitlength}{18.70443764bp}%
    \ifx\svgscale\undefined%
      \relax%
    \else%
      \setlength{\unitlength}{\unitlength * \real{\svgscale}}%
    \fi%
  \else%
    \setlength{\unitlength}{\svgwidth}%
  \fi%
  \global\let\svgwidth\undefined%
  \global\let\svgscale\undefined%
  \makeatother%
  \begin{picture}(1,0.91836368)%
    \put(0,0){\includegraphics[width=\unitlength,page=1]{orient_resolution.pdf}}%
  \end{picture}%
\endgroup%
} \ldots \raisebox{-5pt}{
\begingroup%
  \makeatletter%
  \providecommand\color[2][]{%
    \errmessage{(Inkscape) Color is used for the text in Inkscape, but the package 'color.sty' is not loaded}%
    \renewcommand\color[2][]{}%
  }%
  \providecommand\transparent[1]{%
    \errmessage{(Inkscape) Transparency is used (non-zero) for the text in Inkscape, but the package 'transparent.sty' is not loaded}%
    \renewcommand\transparent[1]{}%
  }%
  \providecommand\rotatebox[2]{#2}%
  \ifx\svgwidth\undefined%
    \setlength{\unitlength}{18.70443764bp}%
    \ifx\svgscale\undefined%
      \relax%
    \else%
      \setlength{\unitlength}{\unitlength * \real{\svgscale}}%
    \fi%
  \else%
    \setlength{\unitlength}{\svgwidth}%
  \fi%
  \global\let\svgwidth\undefined%
  \global\let\svgscale\undefined%
  \makeatother%
  \begin{picture}(1,0.91836368)%
    \put(0,0){\includegraphics[width=\unitlength,page=1]{orient_resolution.pdf}}%
  \end{picture}%
\endgroup%
}}
\newcommand{\nonorientResOrbit}{
\begingroup%
  \makeatletter%
  \providecommand\color[2][]{%
    \errmessage{(Inkscape) Color is used for the text in Inkscape, but the package 'color.sty' is not loaded}%
    \renewcommand\color[2][]{}%
  }%
  \providecommand\transparent[1]{%
    \errmessage{(Inkscape) Transparency is used (non-zero) for the text in Inkscape, but the package 'transparent.sty' is not loaded}%
    \renewcommand\transparent[1]{}%
  }%
  \providecommand\rotatebox[2]{#2}%
  \ifx\svgwidth\undefined%
    \setlength{\unitlength}{18.70443764bp}%
    \ifx\svgscale\undefined%
      \relax%
    \else%
      \setlength{\unitlength}{\unitlength * \real{\svgscale}}%
    \fi%
  \else%
    \setlength{\unitlength}{\svgwidth}%
  \fi%
  \global\let\svgwidth\undefined%
  \global\let\svgscale\undefined%
  \makeatother%
  \begin{picture}(1,0.91836368)%
    \put(0,0){\includegraphics[width=\unitlength,page=1]{nonorient_resolution.pdf}}%
  \end{picture}%
\endgroup%
 \ldots 
\begingroup%
  \makeatletter%
  \providecommand\color[2][]{%
    \errmessage{(Inkscape) Color is used for the text in Inkscape, but the package 'color.sty' is not loaded}%
    \renewcommand\color[2][]{}%
  }%
  \providecommand\transparent[1]{%
    \errmessage{(Inkscape) Transparency is used (non-zero) for the text in Inkscape, but the package 'transparent.sty' is not loaded}%
    \renewcommand\transparent[1]{}%
  }%
  \providecommand\rotatebox[2]{#2}%
  \ifx\svgwidth\undefined%
    \setlength{\unitlength}{18.70443764bp}%
    \ifx\svgscale\undefined%
      \relax%
    \else%
      \setlength{\unitlength}{\unitlength * \real{\svgscale}}%
    \fi%
  \else%
    \setlength{\unitlength}{\svgwidth}%
  \fi%
  \global\let\svgwidth\undefined%
  \global\let\svgscale\undefined%
  \makeatother%
  \begin{picture}(1,0.91836368)%
    \put(0,0){\includegraphics[width=\unitlength,page=1]{nonorient_resolution.pdf}}%
  \end{picture}%
\endgroup%
}

\newcommand{\lk}{\operatorname{lk}}

\newcommand{\cc}{\operatorname{CKh}}

\newcommand{\kh}{\operatorname{Kh}}

\newcommand{\khp}{\operatorname{KhP}}
\newcommand{\jones}{\operatorname{J}}
\newcommand{\diffJones}{\operatorname{DJ}}

\newcommand{\bnbracket}[1]{[\kern-1.5pt [ #1 ]\kern-1.5pt]_{\operatorname{BN}}}
\newcommand{\khbracket}[1]{[\kern-1.5pt [ #1 ]\kern-1.5pt]_{\operatorname{Kh}}}

\newcommand{\One}{\operatorname{1 \kern-3.75pt 1}}

\title{Equivariant Jones Polynomials of periodic links}
\author{Wojciech Politarczyk}

\begin{document}

\maketitle

\begin{abstract}
This paper continues the study of periodic links started in \cite{Politarczyk2}. It contains a 
study of the equivariant analogues of the Jones polynomial, which can be obtained from the 
equivariant Khovanov homology. In this paper we describe basic properties of such polynomials, show 
that they satisfy an analogue of the skein relation and develop a state-sum formula. The skein 
relation in the equivariant case is used to strengthen the periodicity criterion of Przytycki from 
\cite{Przytycki}. The state-sum formula is used to reproved the classical congruence of Murasugi 
from \cite{Murasugi1}.
\end{abstract}

\section{Introduction}

A link is periodic if it possesses certain rotational symmetry of finite order. To be more precise, 
periodic links, are invariant under some semi-free action of a cyclic group $\z{n}$ on $S^{3}$. Due 
to the resolution of the Smith Conjecture, see \cite{BassMorgan}, this property can be restated in 
the following way. A link is $n$-periodic, if it is invariant under a rotation of $\mathbb{R}^{3}$ 
of order $n$. Additionally, we require the link to be disjoint from the axis of the rotation.

The relevance of periodic links stems from the fact, that according to \cite{PrzytyckiSokolov}, 
many finite order symmetries of $3$-manifolds come from symmetries of their Kirby diagrams, i.e., 
if a $3$-manifolds admits an action of a cyclic group of finite order with the fixed point 
homeomorphic to a circle, then it is a result of a Dehn surgery on a periodic link. Additionally, 
periodic links should give us an insight into the theory of skein modules of branched and 
unbranched covers.

Link polynomials of periodic links have been studied by many authors. For example 
\cites{DavisLivingston,Murasugi2} study the Alexander polynomial of periodic links. The first 
paper, 
which studies the Jones polynomial of periodic links is \cite{Murasugi1}, where the author obtained 
a congruence which gives a relation between the Jones polynomial of a periodic link and a 
Jones polynomial of its quotient in terms of a certain congruence. This criterion proved to be very 
efficient in verifying whether a given link is periodic or not. In \cite{Traczyk}, another 
periodicity criterion is given. This criterion gives another congruence for the Jones polynomial. 
This criterion was later generalized in \cites{Przytycki,Yokota}. A survey of the results on the 
Jones polynomials of periodic links can be found in \cite{Przytycki3}

This paper is a continuation of \cite{Politarczyk2}, where the author developed a link homology 
theory, called the Khovanov homology, which is an adaptation of the Khovanov homology to the 
equivariant setting of periodic links. In this paper, instead of studying the homology, we study 
analogues of the Jones polynomial, which can be obtained from the equivariant Khovanov homology. We 
utilize the properties of the homology theory described in \cite{Politarczyk2} to obtain properties 
of the equivariant Jones polynomials, which are further used to obtain some periodicity criteria 
for the Jones polynomial.

In \cite{Politarczyk2} the rational equivariant Khovanov homology of an $n$-periodic link $L$ was 
defined as a triply-graded $\qqq$-vector space
$$\kh_{\z{n}}^{\ast,\ast,\ast}(L;\qqq),$$
where the third grading is supported only for dimensions $d$ such that $d \mid n$ and $d>0$. In 
fact, for any $d \mid n$, the vector space
$$\kh_{\z{n}}^{\ast,\ast,d}(L;\qqq)$$
is a vector space over the larger field $\cyclotomicField{d}$, i.e. the $d$-th cyclotomic field. 
This structure is taken into account in the definition of the equivariant Jones polynomials.
$$\jones_{n,d}(L) = \sum_{i,j} t^{i}q^{j} \dim_{\cyclotomicField{d}} \kh_{\z{n}}^{i,j,d}(L;\qqq).$$
Theorem \ref{thm:properties_equivariant_polynomials} summarizes basic properties of these 
polynomials.

When $p$ is a prime and we study $p^{n}$-periodic links, it is better to consider the following 
difference Jones polynomials
$$\diffJones_{n,s}(L) = \jones_{p^{n},p^{s}}(L) - \jones_{p^{n},p^{s+1}}(L),$$
for $0 \leq s \leq n$. It turns out, that these polynomials have much better properties than the 
equivariant Jones polynomials. The first main result of this paper is the following theorem.

\begin{myThm}{\ref{thm:skein_relations_equiv_jones_pol}}
The difference Jones polynomials have the following properties
\begin{enumerate}
\item $\diffJones_{0}$ satisfies the following version of the skein relation
\begin{align*}
&q^{-2p^n} \diffJones_{n,0}\left(\posOrbit\right) - q^{2p^n} 
\diffJones_{n, 0}\left(\negOrbit\right) = \\
&= \left(q^{-p^n} - q^{p^n}\right) \diffJones_{n,0}\left(\orientResOrbit\right),
\end{align*}
where $\posOrbit$, $\negOrbit$ and $\orientResOrbit$ denote the orbit of positive, negative and 
orientation preserving resolutions of crossing, respectively.
\item for any $0 \leq s \leq n$, $\diffJones_{s}$ satisfies the following congruences
\begin{align*}
&q^{-2p^{n}}\diffJones_{n,n-s}\left(\posOrbit\right) - 
q^{2p^{n}}\diffJones_{n,n-s}\left(\negOrbit\right) \equiv \\
&\equiv \left(q^{-p^n} - q^{p^n}\right) \diffJones_{n,n-s}\left(\orientResOrbit\right) 
\pmod{q^{p^{s}} - q^{-p^{s}}}.
\end{align*}
\end{enumerate}
\end{myThm}
In other words, difference Jones polynomials satisfy an analogue of the skein relation of the 
classical Jones polynomial. The above theorem can be used to easily recover the periodicity 
criterion from \cite{Przytycki}. However, if we use one other property of the equivariant Khovanov 
homology, we can strengthen this criterion considerably.

\begin{myThm}{\ref{thm: strength Przytycki thm}}
Suppose that $L$ is a $p^n$-periodic link and for all $i,j$ we have $\dim_{\qqq} \kh^{i,j}(L;\qqq) 
< \varphi(p^{s})$, then the following congruence holds
$$\jones(L)(q) \equiv \jones(L)(q^{-1}) \pmod{\mathcal{I}_{p^{n},s}},$$
where $\mathcal{I}_{p^{n},s}$ is the ideal generated by the following monomials
$$q^{p^{n}} - q^{-p^{n}}, p\left(q^{p^{n-1}} - q^{-p^{n-1}}\right), \ldots, p^{s-1} 
\left(q^{p^{n-s+1}} - q^{-p^{n-s+1}}\right)$$
and $\jones$ denotes the unreduced Jones polynomial.
\end{myThm}

\begin{ex}\label{ex: congruence 10_61}
Consider the $10_{61}$ knot from the Rolfsen table \cite{Rolfsen}. If we are 
interested in the symmetry of order $5$, then according to SAGE \cite{Sage}, the following 
congruence holds
$$\jones(10_{61})(q) - \jones(10_{61})(q^{-1}) \equiv 0 \pmod{q^{5} -q^{-5}, 5(q - q^{-1})}.$$
Hence, Przytycki's Theorem does not obstruct $10_{61}$ to have symmetry of order $5$. However, if 
we notice that, as depicted on Figure \ref{fig: ranks kh 10_61}, the dimension of 
$\kh^{i,j}(10_{61})$ is always smaller that $\varphi(5) = 4$ and apply Theorem \ref{thm: strength 
Przytycki thm} we obtain
$$\jones(10_{61})(q) - \jones(10_{61})(q^{-1}) \not\equiv 0 \pmod{q^{5} -q^{-5}}$$
Consequently $10_{61}$ is not $5$-periodic.
\end{ex}

\begin{figure}
\centering
\begin{tikzpicture}[scale=.66]
  \draw[->] (0,0) -- (11.5,0) node[right] {$i$};
  \draw[->] (0,0) -- (0,12.5) node[above] {$j$};
  \draw[step=1] (0,0) grid (11,12);
  \draw (5.750,-0.8) node[below] {$\kh$};
  \draw (0.5,-.2) node[below] {$-4$};
  \draw (1.5,-.2) node[below] {$-3$};
  \draw (2.5,-.2) node[below] {$-2$};
  \draw (3.5,-.2) node[below] {$-1$};
  \draw (4.5,-.2) node[below] {$0$};
  \draw (5.5,-.2) node[below] {$1$};
  \draw (6.5,-.2) node[below] {$2$};
  \draw (7.5,-.2) node[below] {$3$};
  \draw (8.5,-.2) node[below] {$4$};
  \draw (9.5,-.2) node[below] {$5$};
  \draw (10.5,-.2) node[below] {$6$};
  \draw (-.2,0.5) node[left] {$-5$};
  \draw (-.2,1.5) node[left] {$-3$};
  \draw (-.2,2.5) node[left] {$-1$};
  \draw (-.2,3.5) node[left] {$1$};
  \draw (-.2,4.5) node[left] {$3$};
  \draw (-.2,5.5) node[left] {$5$};
  \draw (-.2,6.5) node[left] {$7$};
  \draw (-.2,7.5) node[left] {$9$};
  \draw (-.2,8.5) node[left] {$11$};
  \draw (-.2,9.5) node[left] {$13$};
  \draw (-.2,10.5) node[left] {$15$};
  \draw (-.2,11.5) node[left] {$17$};
  \draw (0.5, 0.5) node {$1$};
  \draw (1.5, 2.5) node {$1$};
  \draw (2.5, 2.5) node {$3$};
  \draw (3.5, 3.5) node {$1$};
  \draw (3.5, 4.5) node {$3$};
  \draw (4.5, 4.5) node {$3$};
  \draw (4.5, 5.5) node {$2$};
  \draw (5.5, 5.5) node {$3$};
  \draw (5.5, 6.5) node {$2$};
  \draw (6.5, 6.5) node {$2$};
  \draw (6.5, 7.5) node {$3$};
  \draw (7.5, 7.5) node {$2$};
  \draw (7.5, 8.5) node {$2$};
  \draw (8.5, 8.5) node {$1$};
  \draw (8.5, 9.5) node {$2$};
  \draw (9.5, 9.5) node {$1$};
  \draw (9.5, 10.5) node {$1$};
  \draw (10.5, 11.5) node {$1$};
\end{tikzpicture}
\caption{Ranks of $\kh^{i,j}(10_{61})$ according to \cite{KnotKit}.}
\label{fig: ranks kh 10_61}
\end{figure}

The second main theorem of this papers gives a state-sum formula for the difference Jones 
polynomials. This formula shows that in order to obtain the difference polynomial 
$\diffJones_{n,n-m}(L)$ we need to consider only Kauffman states $s \in \mathcal{S}^{p^{v}}(D)$, 
i.e., those Kauffman states, which inherit a symmetry of order $p^{v}$ from $L$, for $m \leq v \leq 
n$.

\begin{myThm}{\ref{thm: state sum formula}}
Let $D$ be a $p^n$-periodic diagram of a link and let $0 \leq m \leq n$. Under this assumptions, 
the following equality holds.
$$\diffJones_{n,n-m}(D) = (-1)^{n_{-}(D)} q^{n_{+}(D) - 2n_{-}(D)} \sum_{m \leq v \leq n} 
\sum_{s \in \mathcal{S}^{p^{v}}(D)} (-q)^{r(s)} \diffJones_{v,v-m}(s).$$
For a Kauffman state $s$ we write $r(s) = r$ if $s \in \mathcal{S}_{r}(D)$, compare Definition 
\ref{defn: Kauffman states of periodic diagram}.
\end{myThm}

The state sum formula for $\diffJones_{n,0}(L)$ is used to reprove the congruence from 
\cite{Murasugi1}.

The paper is organized as follows. Section \ref{sec: equiv kh} contains a summary of results about 
the equivariant Khovanov homology, which are needed in this paper. In section \ref{sec: equiv 
jones} we describe the basic properties of the equivariant Jones polynomials, and difference 
polynomials. Further, we use Theorem \ref{thm:skein_relations_equiv_jones_pol} to generalize 
Przytycki's periodicity criterion. Section \ref{sec: state-sum formula} contains the discussion of 
the state sum formula and its implications. Section \ref{sec: proofs} contains the proofs of 
Theorem \ref{thm:skein_relations_equiv_jones_pol} and Theorem \ref{thm: state sum formula}.

\paragraph*{Convention} In the remainder part of this article we adopt the convention that all 
links are oriented unless stated otherwise. Furthermore, we always use the unreduced 
Jones polynomial.

\paragraph*{Acknowledgments} The author is very grateful to Józef Przytycki for bringing 
this problem to his attention, to Maciej Borodzik and Prof. Krzysztof Pawałowski for their 
suggestions and corrections, which improved the presentation of the paper.

\section{Equivariant Khovanov homology}
\label{sec: equiv kh}

This section surveys the results from \cite{Politarczyk2}, which will be needed in this 
paper. We only focus on the rational equivariant homology, since our main focus is on the 
equivariant analogues of the Jones polynomial.

Let us start with the recollection of the definition of a periodic link.

\begin{defn}
Let $n$ be a positive integer, and let $L$ be a link in $S^{3}$. We say that $L$ is $n$-periodic, 
if there exists an action of the cyclic group of order $n$ on $S^3$ satisfying the following 
conditions. 
\begin{enumerate}
\item The fixed point set, denoted by $F$, is the unknot.
\item $L$ is disjoint from $F$.
\item $L$ is a $\z{n}$-invariant subset of $S^3$.
\end{enumerate}
\end{defn}

\begin{figure}
\centering
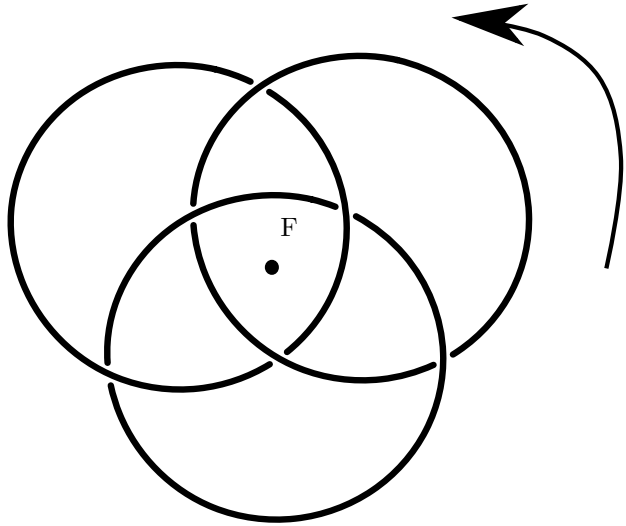
\caption{Borromean rings are $3$-periodic. The fixed point axis $F$ is marked with a dot.}
\label{fig: bor rings 3-periodic}
\end{figure}

\begin{ex}
Borromean rings provide an example of a $3$-periodic link. The symmetry is visualized on Figure 
\ref{fig: bor rings 3-periodic}. The dot marks the fixed point axis.
\end{ex}

\begin{ex}\label{ex: periodic torus links}
Torus links constitute an infinite family of periodic links. In fact, according to \cite{Murasugi2},
the torus link $T(m,n)$ is $d$-periodic if, and only if, $d$ divides either $m$ or $n$.
\end{ex}

Periodic diagrams of periodic links can be described in terms of planar algebras. For the 
definition of planar algebras see \cite{BarNatan}. Take an $n$-periodic planar diagram $D_{n}$ 
with $n$ input disks, like the one on Figure \ref{fig: planar diagram periodic}. Choose a tangle 
$T$ which possesses enough endpoints, and glue $n$ copies of $T$ into the input disks of $D_{n}$. 
In this way, we obtain a periodic link whose quotient is represented by an appropriate closure of 
$T$. See Figure \ref{fig: planar diagram periodic link} for an example.

\begin{figure}
\centering
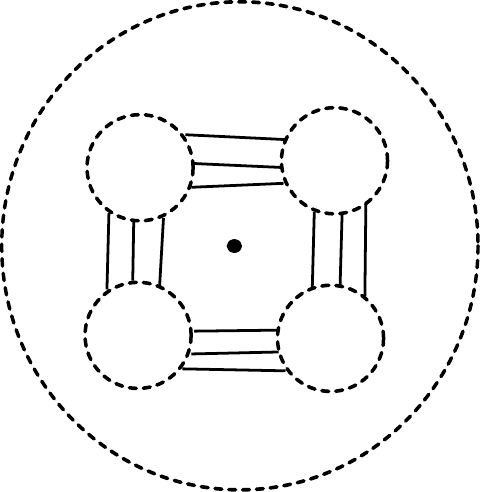
\caption{$4$-periodic planar diagram.}
\label{fig: planar diagram periodic}
\end{figure}

\begin{figure}
\centering
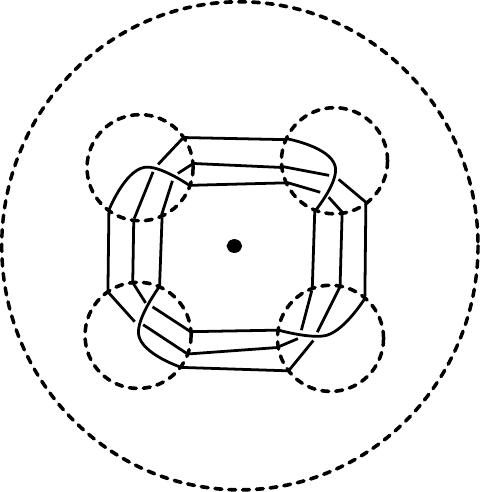
\caption{Torus knot $T(3,4)$ as a $4$-periodic knot obtained from the planar diagram from Figure 
\ref{fig: planar diagram periodic}}
\label{fig: planar diagram periodic link}
\end{figure}

Let $D$ be an $n$-periodic diagram of an $n$ periodic link. The Kauffman states of $D$ can be 
naturally divided into several families according to the type of symmetry they inherit from $D$.

\begin{defn}\label{defn: Kauffman states of periodic diagram}
\begin{enumerate}
\item Let $\mathcal{S}_{r}(D)$ denote the set of Kauffman states of $D$ which were obtained by 
resolving exactly $r$ crossings with the $1$-smoothing.
\item For $d \mid n$, let $\mathcal{S}^{d}(D)$ denote the set of Kauffman states which inherit 
a symmetry of order $d$ from the symmetry of $D$, that is Kauffman states of the form
$$D_{n}(T_1, \ldots, T_{\frac{n}{d}}, T_{1}, \ldots, T_{\frac{n}{d}}, \ldots, T_{1}, \ldots, 
T_{\frac{n}{d}}),$$
where $T_{1}, \ldots, T_{\frac{n}{d}}$ are distinct resolutions of $T$.
\item For a Kauffman state $s$, write $\Iso_{D}(s) = \z{d}$ if, and only if $s \in 
\mathcal{S}^{d}(D)$.
\item Define $\mathcal{S}_{r}^{d}(D) = \mathcal{S}^{d}(D) \cap \mathcal{S}_{r}(D)$.
\item Define $\overline{\mathcal{S}}_{r}^{d}(D)$ to be the quotient of $\mathcal{S}_{r}^{d}(D)$ by 
the action of $\z{n}$.
\end{enumerate}
\end{defn}

\begin{rem}
If $\mathcal{S}^{d}_{r}(D)$ is non-empty, then $d \mid \gcd(n,r)$.
\end{rem}

Analysis of the Khovanov complex $\cc(D;\qqq)$ for an $n$-periodic link diagram $D$ shows that the 
symmetry of the diagram can be lifted to an action of the cyclic group $\z{n}$ on the Khovanov 
complex. This leads to the following conclusion.

\begin{prop}\label{prop: periodic cochain complex}
If $D$ is a periodic link diagram, then $\cc(D;\qqq)$ is a complex of graded 
$\gpAlg{\z{n}}$-modules.
\end{prop}

This leads to the definition of the equivariant Khovanov homology.

\begin{defn}
Define the rational equivariant Khovanov homology of an $n$-periodic diagram $D$ to be the 
following triply-graded module, for which the third grading is supported only for $d \mid n$.
\begin{align*}
\kh_{\z{n}}^{\ast,\ast,d}(D;\qqq) &= 
H^{\ast,\ast}\left(\Hom{}{\cyclicGroupAlg{n}}{\cyclotomicField{d}}{\cc(D; \qqq)}\right) \\
&\cong H^{\ast,\ast}\left(\cyclotomicField{d} \otimes_{\cyclicGroupAlg{n}} \cc(D;\qqq)\right),
\end{align*}
where $\cc(D;\qqq)$ is the Khovanov complex of $D$ with rational coefficients. It is worth to 
notice, that since $\cc(D)$ is a complex of graded modules, Ext groups become also naturally graded, 
provided that we regard $\gpRing{\xi_{d}}$ as a graded module concentrated in degree~$0$.
\end{defn}

The equivariant Khovanov homology depends only on the equivariant isotopy class of the link 
represented by the diagram i.e. we allow deformations of the diagram which commute with the action 
of the cyclic group. Such deformations can be realized as a composition of equivariant Reidemeister 
moves.

\begin{thm}\label{thm:invariance_equivariant_homology}
Equivariant Khovanov homology groups are invariants of periodic links, that is they are invariant 
under equivariant Reidemeister moves.
\end{thm}

Semi-simplicity of the group algebra has a number of implications regarding the structure of the 
rational equivariant Khovanov homology.

\begin{prop}\label{prop: structure rational equiv kh}
Let $D$ be an $n$-periodic link diagram.
\begin{enumerate}
\item The equivariant Khovanov homology can be computed in the following way
\begin{align*}
\kh_{\z{n}}^{\ast,\ast,d}(D; \qqq) &\cong \cyclotomicField{d} \otimes_{\cyclicGroupAlg{n}} 
\kh^{\ast,\ast}(D;\qqq) \\
&\cong e_{d}\kh^{\ast,\ast}(D;\qqq),
\end{align*}
where $e_{d} \in \cyclicGroupAlg{n}$ is a central idempotent, which acts as an identity on an 
irreducible summand of isomorphic to $\cyclotomicField{d}$ and annihilates the complementary 
summand.
\item There exists a decomposition of $\cyclicGroupAlg{n}$-modules
$$\kh^{\ast,\ast}(D;\qqq) = \bigoplus_{d \mid n} \kh_{\z{n}}^{\ast,\ast,d}(D;\qqq).$$
\item If for any $i,j$
$$\dim \kh^{i,j}(D) < \varphi(d),$$
where $\varphi$ is the Euler's totient functions, then
$$\kh_{\z{n}}^{\ast,\ast,d}(D;\qqq) = 0.$$
\end{enumerate}
\end{prop}
\begin{proof}
The first two statements are easy consequences of the existence of the Wedderburn decomposition of 
the group algebra and of the Schur's Lemma. The third one follows from the second, because 
$\kh_{\z{n}}^{\ast,\ast,d}(D;\qqq)$ is a $\cyclotomicField{d}$ vector spaces and
$$\dim \cyclotomicField{d} = \varphi(d).$$
\end{proof}

As it was shown in \cite{Politarczyk2} these properties can be used to determine the equivariant 
Khovanov homology in some cases.

\begin{cor}
Let $T(n,2)$ be the torus link. Let $d > 2$ be a divisor of $n$. According to Example
\ref{ex: periodic torus links}, $T(n,2)$ is $d$-periodic. Let $d' > 2$ and $d' \mid d$.
$$\kh^{\ast,\ast,d'}_{\z{d}}(T(n,2); \qqq) = 0.$$
\end{cor}
\begin{proof}
Indeed, because according to \cite{Khovanov1}*{Prop. 35} for all $i,j$ we have
$$\dim_{\qqq} \kh^{i,j}(T(n,2); \qqq) \leq 1$$
and $\varphi(d') > 1$ if $d' > 2$.
\end{proof}

\begin{cor}
Let $\gcd(3,n) = 1$. The $3$-equivariant Khovanov homology 
$\kh^{\ast,\ast,3}_{\z{3}}(T(n,3);\qqq)$ of the torus knot $T(n,3)$ vanishes.

\noindent If $d > 2$ divides $n$, $d' > 2$ and $d' \mid d$, then 
$\kh^{\ast,\ast,d'}_{\z{d}}(T(n,3); \qqq) = 0$.
\end{cor}
\begin{proof}
  Indeed, because \cite{Turner2}*{Thm 3.1} implies that for all $i,j$ we have
  $$\dim_{\qqq}   \kh^{i,j}(T(n,3); \qqq) \le 1,$$
  provided that $\gcd(3,n)=1$.
\end{proof}

Let $T_{kp^{n}+f}$ be a crossingless diagram of the trivial link with $kp^{n}+f$ components, for a 
prime $p$ and non-negative integers $k,n,f$. This link is $p^{n}$-periodic in such a way that its 
components can be divided into two families. The first family contains $kp^{n}$ components, which 
are freely permuted, and the second one contains the remainder components which are linked with the 
fixed point axis and are rotated by the diffeomorphism generating the action of the cyclic group. 
In order to state the result of the computation of the equivariant Khovanov homology of 
$T_{kp^{n}+f}$, let us introduce some notation.

\begin{defn}\label{defn:diff polynomials trivial links}
Define a sequence of Laurent polynomials i.e. elements of the ring $\gpRing{q,q^{-1}}$.
\begin{align*}
\mathcal{P}_{0}(q) &= q+q^{-1}\\
\mathcal{P}_{n}(q) &= \frac{1}{p^n}\sum_{\stackrel{1 \leq k \leq 
p^n - 1}{\gcd(k,p^n)=1}}\binom{p^n}{k} q^{2k - p^n} + \\
&+ \frac{1}{p^n}\sum_{1 \leq s < n} \sum_{\stackrel{1 \leq k \leq p^n-1}{\gcd(k,p^n) = p^s}} 
\left(\binom{p^n}{k} - \binom{p^{n-s}}{k'}\right) q^{2k - p^n},
\end{align*}
where $k' = k/p^{s}$ and $n \geq 1$.
\end{defn}

\begin{defn}
If $M^{\ast}$ is a graded $\qqq$ vector space, then define its quantum dimension as in 
\cite{Turner2}.
$$\qdim M^{\ast} = \sum_{i} q^{i} \dim M^{i} \in \gpRing{q,q^{-1}}.$$
\end{defn}

\begin{defn}
Let $M_{s}^{k,f}$ be a graded $\qqq$ vector space whose quantum dimension satisfies the following 
equality
\begin{align*}
&\qdim M_{s}^{k,f} = \\
&=(q+q^{-1})^{f}\sum_{\ell = 1}^{k} p^{s \cdot( \ell - 1)} 
\mathcal{P}_{s}(q^{p^{n-s}})^{\ell}  \sum_{\stackrel{0 \leq i_0, \ldots, i_{s-1} \leq k}{i_0 + 
\ldots + i_{s-1} = k - \ell}} \prod_{j=0}^{s-1} \left( p^j \mathcal{P}_{j}(q^{p^{n-j}}) 
\right)^{i_j}.
\end{align*}
\end{defn}

\begin{prop}\label{prop: equiv kh trivial links}
The rational Khovanov homology of the trivial link $T_{kp^{n}+f}$, for some prime $p$, which 
possesses $k$ free orbits of components and $f$ fixed circles, is given by the following formula
$$\kh_{\z{p^{n}}}^{0,\ast,p^{n-u}}(T_{kp^{n}+f}; \qqq) = \bigoplus_{s=n-u}^{n} 
(M_{s}^{k,f})^{\varphi(p^{n-u})}.$$
\end{prop}

\cite{Politarczyk2} contains also a construction of a spectral sequence converging to the 
equivariant Khovanov homology, which is a substitute for the long exact sequence of Khovanov 
homology. First, however, we need to introduce some notation.

Start with a link $L$ and its $n$-periodic diagram $D$. Choose a subset of crossings $X 
\subset \cross{D}$.

\begin{defn}
Let $\alpha \colon \cross(D) \to \{0,1,x\}$ be a map.
\begin{enumerate}
\item If $i \in \{0,1,x\}$ define $|\alpha|_{i} = \#\alpha^{-1}(i)$. 
\item Define the support of $\alpha$ to be $\supp \alpha = \alpha^{-1}(\{0,1\})$.
\item Define also the following family of maps
\begin{align*}
\mathcal{B}_k(X) &= \{\alpha \colon \cross(D) \to \{0,1,x\} \quad | \quad \supp \alpha = X, \quad 
|\alpha|_{1} 
= k \}.
\end{align*}
\item  Denote by $D_{\alpha}$ the diagram obtained from $D$ by resolving crossings from 
$\alpha^{-1}(0)$ by $0$-\hyphenation{smooth-ing} and from $\alpha^{-1}(1)$ by $1$-smoothing. 
\end{enumerate}
\end{defn}

\begin{defn}
If $D$ is a link diagram, $X \subset \cross{D}$ and $\alpha \in \mathcal{D}_{k}(X)$, for some $k$, 
define
$$c(D_{\alpha}) = n_{-}(D_{\alpha}) - n_{-}(D),$$
where $n_{-}(D_{\alpha})$ denotes the number of negative crossings of $D_{\alpha}$.
\end{defn}

\begin{defn}
Let $M^{\ast,\ast}$ be a bi-graded module. For $m,n \in \zzz$ define a new bi-graded module 
$M[n]\{m\}$ in the following way.
$$M[n]\{m\}^{i,j} = M^{i-n,j-m}.$$
\end{defn}

The construction of the spectral sequence starts with a choice of a single orbit of crossings of a 
periodic diagram $D$. To preserve the symmetry, we have to consider all resolutions of the 
crossings from the chosen orbit. Khovanov complex of $D$ can be made into a bicomplex whose columns 
are expressible in terms of the Khovanov complexes of the diagrams obtained from $D$ by resolving 
only crossings from the chosen orbit. This leads to the following spectral sequence.

\begin{thm}\label{thm:spectral_sequence_equivariant_homology}
Let $L$ be a $p^{n}$-periodic link, where $p$ is an odd prime, and let $X \subset \cross{D}$ 
consists of a single orbit. Under this assumption, for any $0 \leq s \leq n$ there exists a 
spectral sequence $\{{}_{p^{n-s}}E_{r}^{\ast,\ast}, d_r\}$
of graded modules converging to $\kh_{\z{p^{n}}}^{\ast,\ast,p^{n-s}}(D;\qqq)$ with
\begin{align*}
{}_{p^{n-s}}E^{0,j}_{1} &= \kh_{\z{p^n}}^{j,\ast,p^{n-s}}(D_{\alpha_0};\qqq)[c(D_{\alpha_0})]\{ 
q(\alpha_{0}) \}, \\
{}_{p^{n-s}}E^{p^{n}, j}_{1} &= \kh_{\z{p^n}}^{j, \ast, p^{n-s}} (D_{\alpha_{1}};\qqq) 
[c(D_{\alpha_1})] \{ q(\alpha_{1}) \},\\
{}_{p^{n-s}}E^{i,j}_{1} &= \bigoplus_{0 \leq v \leq u_{i}} \bigoplus_{\alpha \in 
\overline{\mathcal{B}}_{i}^{p^{v}}(X)} \kh_{\z{p^{v}}}^{j, \ast, 
k(v,s)}(D_{\alpha};\qqq)[c(D_{\alpha})]\{q(\alpha)\}^{\ell(v,s)}
\end{align*}
for $0 < i < p^{n}$. Above we used the following notation $i = p^{u_{i}} g$, 
where $\gcd(p,g) = 1$ and $\alpha_{0}$, $\alpha_{1}$ are the unique elements of $\mathcal{B}_0(X)$ 
and $\mathcal{B}_{p^{n}}(X)$, respectively, and
\begin{align*}
q(\alpha) &= i + 3c(D_{\alpha}) + p^{n}, \\
k(s,v) &= \left\{
\begin{array}{ll}
1, & v \leq s, \\
p^{v-s}, & v > s, \\
\end{array}
\right. \\
\ell(s,v) &= \left\{
\begin{array}{ll}
\varphi(p^{n-s}), & v \leq s, \\
p^{n-v}, & v > s, \\
\end{array}
\right.
\end{align*}
\end{thm}

\section{Equivariant Jones polynomials}
\label{sec: equiv jones}

Analogously as in the classical case, we can define the equivariant Jones polynomials of a periodic 
link.

\begin{defn}
Let $L$ be an $n$-periodic link. For $d \mid n$ define a $d$-th equivariant Khovanov polynomial by
$$\khp_{n,d}(L)(t,q) = \sum_{i,j}t^{i} q^{j} \dim_{\gpAlg{\xi_d}} \kh^{i,j,d}(L; \qqq)$$
and $d$-th equivariant Jones polynomial
$$\jones_{n,d}(L)(q) = \khp_{n,d}(L)(-1,q).$$
\end{defn}

\noindent The next theorem describes basic properties of equivariant Jones polynomials.

\begin{thm}\label{thm:properties_equivariant_polynomials}
Let $L$ be an $n$-periodic link and let $d \mid n$.
\begin{enumerate}
\item Equivariant Khovanov and Jones polynomials are invariants of periodic links i.e. they are 
invariant under Reidemeister moves.
\item If $\jones(L)$ denotes the ordinary unreduced Jones polynomial, then the following equality 
holds.
$$\jones(L) = \sum_{d \mid n} \phi(d) \jones_{n,d}(L),$$
where $\phi$ denotes the Euler's totient function.
\item If $d \mid n$ and for all $i,j$ we have $\dim_{\qqq}\kh^{i,j}(L; \qqq) < \varphi(d)$, then
\begin{align*}
\khp_{n,d}(L) &= 0, \\
\jones_{n,d}(L) &= 0.
\end{align*}
\end{enumerate}
\end{thm}

\begin{proof}
The first part of the theorem follows from theorem \ref{thm:invariance_equivariant_homology}. The 
remaining two are consequences of Proposition \ref{prop: structure rational equiv kh}.
\end{proof}

\noindent From now on we assume that all links are $p^{n}$-periodic, for $p$ an odd prime and $n > 
0$.

\begin{defn}\label{defn:difference Jones polynomials}
Suppose that $D$ is a $p^n$-periodic link diagram. Define the difference Jones 
polynomials
$$\diffJones_{n,s}(D) = \jones_{p^{n},p^s}(D) - \jones_{p^{n},p^{s+1}}(D)$$
for $0 \leq s \leq n$.
\end{defn}

\begin{cor}\label{cor:Jones_decomposition}
The following equality holds.
$$\jones(D) = \sum_{s = 0}^{n} p^{s}\diffJones_{n,s}(D)$$ 
\end{cor}
\begin{proof}
Proof follows from theorem \ref{thm:properties_equivariant_polynomials} and a simple fact that
$$\phi(p^s) = p^s - p^{s-1}.$$
\end{proof}

\begin{ex}\label{ex: diffJones trivial links}
Let $T_{k\cdot p^{n} + f}$ be as in Proposition \ref{prop: equiv kh trivial links}. Proposition 
\ref{prop: equiv kh trivial links} implies that the equivariant and 
difference Jones polynomials of $T_{k\cdot p^{n} + f}$ can be expressed in terms 
of polynomials $\mathcal{P}_{s}$ from Definition \ref{defn:diff polynomials trivial links}.
\begin{align*}
\jones_{p^{n}, p^{n-u}}(T_{k \cdot p^{n} + f}) &= \sum_{s = n-u}^{n} \qdim M^{k,f}_{s}, \\
\diffJones_{n,n-u}(T_{k \cdot p^{n} + f}) &= \qdim M^{k,f}_{n-u}.
\end{align*}
\end{ex}

One of the most important properties of the Jones polynomials is the skein relation
$$q^{-2} \jones\left( \raisebox{-5pt}{
\begingroup%
  \makeatletter%
  \providecommand\color[2][]{%
    \errmessage{(Inkscape) Color is used for the text in Inkscape, but the package 'color.sty' is not loaded}%
    \renewcommand\color[2][]{}%
  }%
  \providecommand\transparent[1]{%
    \errmessage{(Inkscape) Transparency is used (non-zero) for the text in Inkscape, but the package 'transparent.sty' is not loaded}%
    \renewcommand\transparent[1]{}%
  }%
  \providecommand\rotatebox[2]{#2}%
  \ifx\svgwidth\undefined%
    \setlength{\unitlength}{18.70443764bp}%
    \ifx\svgscale\undefined%
      \relax%
    \else%
      \setlength{\unitlength}{\unitlength * \real{\svgscale}}%
    \fi%
  \else%
    \setlength{\unitlength}{\svgwidth}%
  \fi%
  \global\let\svgwidth\undefined%
  \global\let\svgscale\undefined%
  \makeatother%
  \begin{picture}(1,0.91836368)%
    \put(0,0){\includegraphics[width=\unitlength,page=1]{crossing_pos.pdf}}%
  \end{picture}%
\endgroup%
} \right) - q^{2} \jones \left( \raisebox{-5pt}{
\begingroup%
  \makeatletter%
  \providecommand\color[2][]{%
    \errmessage{(Inkscape) Color is used for the text in Inkscape, but the package 'color.sty' is not loaded}%
    \renewcommand\color[2][]{}%
  }%
  \providecommand\transparent[1]{%
    \errmessage{(Inkscape) Transparency is used (non-zero) for the text in Inkscape, but the package 'transparent.sty' is not loaded}%
    \renewcommand\transparent[1]{}%
  }%
  \providecommand\rotatebox[2]{#2}%
  \ifx\svgwidth\undefined%
    \setlength{\unitlength}{18.70443764bp}%
    \ifx\svgscale\undefined%
      \relax%
    \else%
      \setlength{\unitlength}{\unitlength * \real{\svgscale}}%
    \fi%
  \else%
    \setlength{\unitlength}{\svgwidth}%
  \fi%
  \global\let\svgwidth\undefined%
  \global\let\svgscale\undefined%
  \makeatother%
  \begin{picture}(1,0.91836368)%
    \put(0,0){\includegraphics[width=\unitlength,page=1]{crossing_neg.pdf}}%
  \end{picture}%
\endgroup%
} \right) = \left( q^{-1} - q
\right) \jones\left( \raisebox{-5pt}{
\begingroup%
  \makeatletter%
  \providecommand\color[2][]{%
    \errmessage{(Inkscape) Color is used for the text in Inkscape, but the package 'color.sty' is not loaded}%
    \renewcommand\color[2][]{}%
  }%
  \providecommand\transparent[1]{%
    \errmessage{(Inkscape) Transparency is used (non-zero) for the text in Inkscape, but the package 'transparent.sty' is not loaded}%
    \renewcommand\transparent[1]{}%
  }%
  \providecommand\rotatebox[2]{#2}%
  \ifx\svgwidth\undefined%
    \setlength{\unitlength}{18.70443764bp}%
    \ifx\svgscale\undefined%
      \relax%
    \else%
      \setlength{\unitlength}{\unitlength * \real{\svgscale}}%
    \fi%
  \else%
    \setlength{\unitlength}{\svgwidth}%
  \fi%
  \global\let\svgwidth\undefined%
  \global\let\svgscale\undefined%
  \makeatother%
  \begin{picture}(1,0.91836368)%
    \put(0,0){\includegraphics[width=\unitlength,page=1]{orient_resolution.pdf}}%
  \end{picture}%
\endgroup%
} \right).$$
The skein relation is a consequence of the long exact sequence of Khovanov homology. Spectral 
sequences from Theorem \ref{thm:spectral_sequence_equivariant_homology} can be utilized to obtain 
the following analogue of the skein relation for the difference Jones polynomials.

\begin{thm}\label{thm:skein_relations_equiv_jones_pol}
The difference Jones polynomials have the following properties
\begin{enumerate}
\item $\diffJones_{0}$ satisfies the following version of the skein relation
\begin{align*}
&q^{-2p^n} \diffJones_{n,0}\left(\posOrbit\right) - q^{2p^n} 
\diffJones_{n, 0}\left(\negOrbit\right) = \\
&= \left(q^{-p^n} - q^{p^n}\right) \diffJones_{n,0}\left(\orientResOrbit\right),
\end{align*}
where $\posOrbit$, $\negOrbit$ and $\orientResOrbit$ denote the orbit of positive, negative and 
orientation preserving resolutions of crossing, respectively.
\item for any $0 \leq s \leq n$, $\diffJones_{s}$ satisfies the following congruences
\begin{align*}
&q^{-2p^{n}}\diffJones_{n,n-s}\left(\posOrbit\right) - 
q^{2p^{n}}\diffJones_{n,n-s}\left(\negOrbit\right) \equiv \\
&\equiv \left(q^{-p^n} - q^{p^n}\right) \diffJones_{n,n-s}\left(\orientResOrbit\right) 
\pmod{q^{p^{s}} - q^{-p^{s}}}.
\end{align*}
\end{enumerate}
\end{thm}

\begin{rem}
We defer the proof of Theorem \ref{thm:skein_relations_equiv_jones_pol} to Section 
\ref{sec: proofs}.
\end{rem}

The above theorem has a number of consequences regarding the Jones polynomial of a periodic link. 
For example, it enables us to write down a few criteria for the periodicity of a link in terms of 
its Jones polynomial. One such example was given by J.H. Przytycki in \cite{Przytycki}.

\begin{thm}\label{thm: przytycki thm}
Suppose that $L$ is a $p^n$-periodic link. Then the following congruence holds
$$\jones(L)(q) \equiv \jones(L)(q^{-1}) \pmod{\mathcal{I}_{p^n}},$$
where $\mathcal{I}_{p^n}$ is an ideal generated by the following monomials 
$$p^n, p^{n-1} \left(q^{p} - q^{-p} \right), \ldots, p\left(q^{p^{n-1}} - q^{p^{-n-1}}\right), 
q^{p^{n}} - q^{-p^{n}}.$$
\end{thm}
\begin{proof}
Notice that
\begin{align*}
&\jones\left(\posOrbit\right) - \jones\left(\negOrbit\right) \equiv \\
&q^{2p^{n}}\jones\left(\posOrbit\right) - q^{-2p^{n}}\jones\left(\negOrbit\right) \equiv \\
& \sum_{s=0}^{n} p^{n-s}\left(q^{2p^{n}}\diffJones_{n,n-s}\left(\posOrbit\right) - 
q^{-2p^{n}}\diffJones_{n,n-s}\left(\negOrbit\right)\right) \equiv \\
&\equiv 0 \pmod{\mathcal{I}_{p^{n}}}.
\end{align*}
Hence, switching crossings from a single orbit does not change the Jones polynomial modulo 
$\mathcal{I}_{p^{n}}$. Since we can pass from $L$ to its mirror image $L^{!}$ by switching one 
orbit at at time, it follows that
$$\jones(L) \equiv \jones(L^{!}) \pmod{\mathcal{I}_{p^{n}}}.$$
Taking into account the relation between the Jones polynomials of $L$ and $L^{!}$
$$\jones(L^{!})(q) = \jones(L)(q^{-1})$$
concludes the proof
\end{proof}

The above theorem can be considerably strengthened with the aid of Theorem 
\ref{thm:properties_equivariant_polynomials}.

\begin{thm}\label{thm: strength Przytycki thm}
Suppose that $L$ is a $p^n$-periodic link and for all $i,j$ we have $\dim_{\qqq} \kh^{i,j}(L;\qqq) 
< \varphi(p^{s})$, then the following congruence holds
$$\jones(L)(q) \equiv \jones(L)(q^{-1}) \pmod{\mathcal{I}_{p^{n},s}},$$
where $\mathcal{I}_{p^{n},s}$ is the ideal generated by the following monomials
$$q^{p^{n}} - q^{-p^{n}}, p\left(q^{p^{n-1}} - q^{-p^{n-1}}\right), \ldots, p^{s-1} 
\left(q^{p^{n-s+1}} - q^{-p^{n-s+1}}\right).$$
\end{thm}
\begin{proof}
First notice that Theorem \ref{thm:properties_equivariant_polynomials} implies that for $s' \geq s$ 
the equivariant Jones polynomials $\jones_{n,p^{s'}}(L)$ vanish. Therefore $\diffJones_{s'} = 
0$. Corollary \ref{cor:Jones_decomposition} implies that
$$\jones(L) = \sum_{i=0}^{s-1}p^{i}\diffJones_{i}.$$
Now argue as in the proof of the previous Theorem to obtain the desired result.
\end{proof}

\section{State-sum formula}
\label{sec: state-sum formula}

One of the most important properties of the Jones polynomial is that it can be written as a certain 
sum over all Kauffman states obtained from the given diagram. This point of view on the Jones 
polynomial was pioneered in \cite{Kauffman}. In particular, if $D$ is a link diagram, then the 
Jones polynomial of the corresponding links expands into the following sum

$$\jones(D) = (-1)^{n_{-}(D)} q^{n_{+}(D) - 2n_{-}(D)} \sum_{s \in 
\mathcal{S}(D)} (-q)^{r(s)} (q+q^{-1})^{k(s)},$$

\noindent where $k(s)$ denotes the number of components of the Kauffman state $s$ and $r(s)$ 
denotes the number of $1$-smoothings used to get the Kauffman state $s$, see 
\cite{Turner}*{Lect. 1}. As it turns out, it is possible to develop analogous expansions for 
the difference 
Jones polynomials.

\begin{thm}\label{thm: state sum formula}
Let $D$ be a $p^n$-periodic diagram of a link and let $0 \leq m \leq n$. Under this assumptions, 
the following equality holds.
$$\diffJones_{n,n-m}(D) = (-1)^{n_{-}(D)} q^{n_{+}(D) - 2n_{-}(D)} \sum_{m \leq v \leq n} 
\sum_{s \in \mathcal{S}^{p^{v}}(D)} (-q)^{r(s)} \diffJones_{v,v-m}(s).$$
For a Kauffman state $s$ we write $r(s) = r$ if $s \in \mathcal{S}_{r}(D)$, compare Definition 
\ref{defn: Kauffman states of periodic diagram}.
\end{thm}

\begin{rem}
The proof of the above Theorem is deferred to Section \ref{sec: proofs}.
\end{rem}

An application of the state sum formula for the difference Jones polynomials is concerned with 
the following classical periodicity criterion of Murasugi \cite{Murasugi1}.

\begin{thm}
Let $D$ be a $p^{n}$-periodic link diagram. Let $D_{\ast}$ denote the quotient diagram.
$$\jones(D) \equiv \jones(D_{\ast})^{p^{n}} \pmod{p, (q+q^{-1})^{\alpha(D)(p^{n}-1)}-1},$$
where
$$\alpha(D) = \left\{ 
\begin{array}{ll}
1, & 2 \nmid \lk(D,F), \\
2, & 2 \mid \lk(D,F).
\end{array}
\right.$$
Above, $F$ denotes the fixed point set.
\end{thm}
\begin{proof}
Let us analyze the relation between $\diffJones_{n,0}(D)$ and $\jones(D_{\ast})$.

\begin{prop}
The following congruence holds.
$$\diffJones_{n,0}(D) \equiv \jones(D_{\ast})^{p^{n}} \pmod{p, (q+q^{-1})^{\alpha(D)(p^{n}-1)}-1}.$$
\end{prop}
\begin{proof}
Notice that the state formula for $\diffJones_{n,0}(D)$ involves only Kauffman states which inherit 
$\z{p^{n}}$-symmetry. Such Kauffman states correspond bijectively to the Kauffman states of the 
quotient diagram.

Let $s$ be a Kauffman state obtained from $D$ such that $\Iso(s) = \z{p^{n}}$. Assume that $s$ 
consists of $k$ free orbits and $f$ fixed circles. Thus, according to Example \ref{ex: diffJones 
trivial links} the Kauffman state contributes
\begin{align*}
&(-1)^{n_{-}(D)+r(s)} q^{n_{+}(D) - 2n_{-}(D)+r(s)} (q^{p^{n}} + q^{-p^{n}})^{k}(q+q^{-1})^f \equiv 
\\
&\equiv (-1)^{n_{-}(D)+r(s)} q^{n_{+}(D) - 2n_{-}(D)+r(s)} (q+q^{-1})^{kp^{n} + f} \pmod{p}.
\end{align*}
to the state sum for $\diffJones_{n,0}$. The quotient Kauffman state $s_{\ast}$ consists of $k+f$ 
components. Hence, it contributes
$$(-1)^{p^{n} (n_{-}(D_{\ast}) +r(s_{\ast}))} q^{p^{n}(n_{+}(D_{\ast}) - 2n_{-}(D_{\ast}) + 
r(s_{\ast}))} (q+q^{-q})^{p^{n}(k+f)}$$
to $\jones(D_{\ast})^{p^{n}}$ mod $p$. The difference of both contributions is divisible by 
$$(q+q^{-1})^{f(p^{n}-1)}-1.$$
Since
$$f \equiv \lk(D,F) \pmod{2},$$
the proposition follows.
\end{proof}

We can conclude the proof once we note that
$$\jones(D) \equiv \diffJones_{n,1}(D) \pmod{p}$$
by Corollary \ref{cor:Jones_decomposition}.
\end{proof}

\section{Proofs}
\label{sec: proofs}

This section is entirely devoted to the presentation of proofs of Theorems 
\ref{thm:skein_relations_equiv_jones_pol} and \ref{thm: state sum formula}. These two proofs are 
very alike in principle, therefore we only perform detailed calculations in the first proof, 
because the very same calculations are present in the second proof as well.

\subsection{Proof of Theorem \ref{thm:skein_relations_equiv_jones_pol}}
\noindent Let us begin with a definition.

\begin{defn}\label{defn:Poincare_polynomial_spectral_seq}
Let $\{E_{s}^{\ast,\ast}, d_s\}$ be a spectral sequence of graded 
finite-dimensional $\mathbb{F}$-modules, where $\mathbb{F}$ is a field, converging to some 
doubly-graded finite-dimensional $\mathbb{F}$-module $H^{\ast, \ast}$. Suppose that the spectral 
sequence collapses at some finite stage. Define the Poincar\'e polynomial of the $E_{s}$ page to be 
the following polynomial.
$$P(E_{s})(t,q) = \sum_{i,j} t^{i+j} \qdim_{\mathbb{F}} E_s^{i,j}$$
Poincar\'e polynomial admits the following decomposition
$$P(E_{s}) = \sum_{i} t^{i} P_{i}(E_{s}),$$
where
$$P_{i}(E_{s}) = \sum_{j} t^{j} \qdim_{\mathbb{F}} E^{i,j}_{s}.$$
\end{defn}

\begin{lemma}\label{lemma:property_poincare_polynomial_spectral_seq}
For any $s > 0$ the following equality holds, whenever it makes sense,
$$P(E_{s})(-1,q) = P(E_{\infty})(-1,q).$$
\end{lemma}
\begin{proof}
This is a direct consequence of \cite{McCleary}*{Ex. 1.7}.
\end{proof}

\noindent Let us now analyze $E_{1}$ pages of the spectral sequences from Theorem 
\ref{thm:spectral_sequence_equivariant_homology}, for odd $p$. Let us make the following notation. 
If $1 \leq i \leq p^{n}-1$, then $i =  p^{u_{i}} g$, where $\gcd(p,g)=1$. Poincar\'e polynomials of 
columns are given below.

\begin{align}
\label{align: khp 0 column} P_{0}({}_{p^{n-s}}E_{1}) &= t^{c(D_{\alpha_{0}})} 
q^{3c(D_{\alpha_{0}})} 
\khp_{p^{n},p^{n-s}}(D_{\alpha_{0}}), \\
\label{align: khp last column} P_{p^{n}}({}_{p^{n-s}}E_{1}) &= t^{c(D_{\alpha_{1}})} q^{3 
c(D_{\alpha_{1}}) + 2 p^{n}} 
\khp_{p^{n},p^{n-s}}(D_{\alpha_{1}}), \\
\label{align: khp middle columns}P_{i}({}_{p^{n-s}}E_{1}) &= \sum_{0 \leq v \leq \min{(s, u_{i}})} 
\sum_{\alpha \in 
\overline{\mathcal{B}}_{i}^{p^{v}}(X)} t^{c(D_{\alpha})} q^{i + 3c(D_{\alpha}) + p^{n}} 
\khp_{p^{v},p^{1}}(D_{\alpha}) + \nonumber \\
&+ \sum_{\min{(s, u_{i})} < v \leq u_{i}} \sum_{\alpha \in \overline{\mathcal{B}}_{i}^{p^{v}}} 
t^{c(D_{\alpha})} q^{i + 3c(D_{\alpha}) + p^{n}} \khp_{p^{v},p^{v-s}}(D_{\alpha}).
\end{align}

\noindent In order to make further computations more manageable let us introduce the following 
notation.
\begin{align*}
G_{i}(v,w) &= \sum_{\alpha \in \overline{\mathcal{B}}_{i}^{p^{v}}} t^{c(D_{\alpha})} 
q^{3c(D_{\alpha})} \khp_{p^{v},p^{w}}(D_{\alpha}), \\
DJG_{i}(v,w) &= G_{i}(v,w)(-1,q) = \sum_{\alpha \in \overline{\mathcal{B}}_{i}^{p^{v}}} 
(-1)^{c(D_{\alpha})} q^{3c(D_{\alpha})} \diffJones_{v,w}(D_{\alpha})
\end{align*}
so for $1 \leq i \leq p^{n} - 1$ the Poincar\'e polynomial can be expressed as the following more 
compact sum.
\begin{align}
\label{align: khp middle column compact}
P_{i}({}_{p^{n-s}E_{1}}) &= q^{i+p^{n}}\sum_{v = 0}^{\min(s,u_{i})} G_{i}(v,0) + \\
&= q^{i+p^{n}} \sum_{v = \min(s,u_{i}) + 1}^{u_{i}} G_{i}(v,v-s).
\end{align}

\begin{lemma}
Following formula holds for the Poincar\'e polynomials
\begin{align*}
&P({}_{p^{n-s}}E_1) - P({}_{p^{n-s+1}}E_1) = \\ 
&\sum_{1 \leq j \leq p^{n-s}-1} t^{j\cdot p^{s}} q^{j \cdot p^{s} + p^{n}} \sum_{v = s}^{u_i} 
\left(G_{jp^{s}}(v,s-v) - G_{jp^{s}}(v,v-s+1)\right) \\
&+t^{c(D_{\alpha_0})} q^{3c(D_{\alpha_0}) + p^n} \left( \khp_{p^{n},p^{n-s}}(D_{\alpha_0}) - 
\khp_{p^{n},p^{n-s+1}}(D_{\alpha_0}) \right) \\
&+t^{c(D_{\alpha_1}) + p^n} q^{3 c(D_{\alpha_1}) + 2p^n} \left( \khp_{p^{n},p^{n-s}}(D_{\alpha_1}) 
- \khp_{p^{n},p^{n-s+1}}(D_{\alpha_1}) \right).
\end{align*}
\end{lemma}
\begin{proof}
Indeed, because
\begin{align*}
&P_{i}({}_{p^{n-s}}E_{1}) - P_{i}(E_{p^{n-s+1}}E_{1}) = \\
&= \left\{
\begin{array}{ll}
\sum_{v = s}^{u_{i}} \left(G_{i}(v,s-v)) - G_{i}(v,v-s+1)\right), & p^{s} \mid i, \\
0, & p^{s} \nmid i,
\end{array}
\right.
\end{align*}
which can be easily verified using formula (\ref{align: khp middle column compact}).
\end{proof}

\begin{cor}\label{cor:difference_equivariant_jones_polynomials}
The following formula holds for the difference polynomials
\begin{align*}
&\diffJones_{n,n-s}(D) = \\
&\sum_{1 \leq j \leq p^{n-s}-1} (-1)^{j\cdot p^{s}} q^{j \cdot p^{s} + p^{n}} \sum_{v=s}^{u_{i}} 
DJG_{jp^{s}}(v,v-s) \\
&+(-1)^{c(D_{\alpha_0})} q^{3c(D_{\alpha_0}) + p^n} \diffJones_{n-s}(D_{\alpha_{0}}) \\
&+(-1)^{c(D_{\alpha_1}) + p^n} q^{3 c(D_{\alpha_1}) + 2p^n} \diffJones_{n-s}(D_{\alpha_{1}}).
\end{align*}
\end{cor}
\begin{proof}
It follows easily from previous Lemma by substituting $t = -1$ and noting that 
$P({}_{p^{n-s}}E_{1})(-1,q) = \jones_{p^{n},p^{n-s}}(D)$, by Lemma 
\ref{lemma:property_poincare_polynomial_spectral_seq}.
\end{proof}

\begin{defn}
For any $0 \leq i \leq p^n$ define a map
\begin{align*}
&\kappa \colon \mathcal{B}_{i}(X) \to \mathcal{B}_{p^n - i}(X)\\
&\kappa(\beta)(c) =
\left\{
\begin{array}{ll}
1 - \beta(c), & c \in X, \\
\beta(c), & c \notin X. \\
\end{array}
\right.
\end{align*}

\end{defn}

\begin{prop}\label{prop:changing_orbit}
Let $D$ be $p^{n}$-periodic link diagram and let $X \subset \cross{D}$ be a 
chosen orbit of crossings. Suppose that all crossings from $X$ are positive and let $D^{!}$ 
denote invariant link diagram obtained from $D$ by changing all crossings from $X$ to negative 
ones. 
Then the following equalities hold
\begin{align*}
D_{\alpha} &= D^{!}_{\kappa(\alpha)} \\
|\kappa(\alpha)|_{u} &= p^{n} - |\alpha|_{u}, \mbox{ for } u=0,1, \\
c(D_{\alpha}) &= c(D^{!}_{\kappa(\alpha)}) + p^{n} \\
\end{align*}
\end{prop}
\begin{proof}
The first two equalities are direct consequences of definitions. To prove the third one notice that
\begin{align*}
n_{-}(D) = n_{-}(D^{!}) - p^n.
\end{align*}
Therefore
\begin{align*}
&c(D_{\alpha}) = n_{-}(D_{\alpha}) - n_{-}(D) = n_{-}(D^{!}_{\kappa(\alpha)}) - 
n_{-}(D^{!}) + p^n = \\
&= c(D^{!}_{\kappa(\alpha)}) + p^{n}.
\end{align*}
\end{proof}

\begin{proof}[Proof of Theorem \ref{thm:skein_relations_equiv_jones_pol}]
To prove the first part notice that from corollary 
\ref{cor:difference_equivariant_jones_polynomials} it follows that
\begin{align*}
\diffJones_{n,0}(D) &= (-1)^{c(D_{\alpha_0})} q^{3c(D_{\alpha_0}) + p^n}  
\diffJones_{n,0}(D_{\alpha_0}) \\
&+(-1)^{c(D_{\alpha_1}) + p^n} q^{3 c(D_{\alpha_1}) + 2p^n} \diffJones_{n,0}(D_{\alpha_1}),
\end{align*}
because $D_{\alpha_0}$ and $D_{\alpha_{1}}$ are the only diagrams with isotropy group equal to 
$\z{p^n}$, because these are the only invariant diagrams.

Now without loss of generality assume that the chosen orbit of crossings consists of positive 
crossings. Then $D_{\alpha_0}$ inherits orientations from $D$ and therefore $c(D_{\alpha_0}) = 0$. 
Therefore
\begin{align*}
\diffJones_{n,0}\left(\posOrbit\right) &= q^{p^n} \diffJones_{n,0}\left(\orientResOrbit\right) \\
&+(-1)^{c(D_{\alpha_1}) + p^n} q^{3 c(D_{\alpha_1}) + 2p^n} \diffJones_{n,0}(D_{\alpha_1}),
\end{align*}

On the other hand for $D^{!}$ as in the previous proposition $D^{!}_{\alpha_1}$ 
inherits orientation. Furthermore $c(D^{!}_{\alpha_1}) = - p^{n}$. This gives the following 
equality
\begin{align*}
\diffJones_{n,0}\left(\negOrbit\right) &= (-1)^{c(D^{!}_{\alpha_0})} q^{3c(D^{!}_{\alpha_0}) + p^n}
\diffJones_{n,0}(D^{!}_{\alpha_0}) \\
&+ q^{-p^n} \diffJones_{n,0}\left(\orientResOrbit\right),
\end{align*}
Denote $c = c(D_{\alpha_1})$, then $D_{\alpha_1} = D^{!}_{\alpha_0}$ and 
$c(D^{!}_{\alpha_0}) = c - p^{n}$ by Proposition \ref{prop:changing_orbit}. Therefore
\begin{align*}
\diffJones_{n,0}\left(\posOrbit\right) &= q^{p^n} \diffJones_{0}\left(\orientResOrbit\right) \\
&+(-1)^{c + p^n} q^{3 c + 2p^n} \diffJones_{n,0}(D_{\alpha_1}), \\
\diffJones_{n,0}\left(\negOrbit\right) &= (-1)^{c - p^n} q^{3c - 2p^n}  
\diffJones_{n,0}(\overline{D}_{\alpha_0}) \\
&+ q^{-p^n} \diffJones_{0}\left(\orientResOrbit\right).
\end{align*}
From the above equalities the first part of the theorem follows easily.

To prove the second part, notice that Proposition \ref{prop:changing_orbit} implies that for $v 
\geq s$ the following equality holds.
\begin{align*}
&(-1)^{i}q^{i-p^{n}}DJG_{i}(v,v-s)(D) - (-1)^{p^{n}-i}q^{4p^{n}-i}DJG_{i}(v,v-s)(D^{!}) = \\
&= \sum_{\alpha \in \overline{\mathcal{B}}_{i}^{p^{v}}} 
(-1)^{c+p^{n}+1}q^{3c+3p^{n}}\diffJones_{v,v-s}\left( q^{i-p^{n}} - q^{p^{n}-i} \right)(D_{\alpha}).
\end{align*}
Consequently, if $p^{s} \mid i$, then the above difference is divisible by $q^{p^{s}}-q^{-p^{s}}$. 
To finish the proof combine formula \ref{cor:difference_equivariant_jones_polynomials} and with the 
discussion above.
\end{proof}

\subsection{Proof of Theorem \ref{thm: state sum formula}}

\begin{proof}[Proof of Theorem \ref{thm: state sum formula}]
According to the definition, the equivariant Jones polynomials, can be written as the following 
sum.
\begin{align}\label{eqn: equiv Jones pol sum}
&q^{-n_{+}(D) + n_{-}(D)} \jones_{p^{n},p^{n-s}}(D) = \nonumber \\
&= \sum_{r = n_{-}(D)}^{n_{+}(D)} (-q)^{r} 
\qdim_{\cyclotomicField{p^{n-s}}} 
\Hom{}{\cyclicGroupAlg{p^{n}}}{\cyclotomicField{p^{n-s}}}{\cc^{r-n_{-}(D),\ast}(D)}.
\end{align}
Thus, the only thing we need to do is to determine the quantum dimension of the following graded 
module
$$\Hom{}{\cyclicGroupAlg{p^{n}}}{\cyclotomicField{p^{n-s}}}{\cc^{r-n_{-}(D),\ast}(D)}$$
for $0 \leq r \leq n_{+}(D) + n_{-}(D)$.
Performing calculations as in the proof of Theorem 
\ref{thm:spectral_sequence_equivariant_homology}. 
Let $r = p^{u_{r}} g$, where $\gcd(p,g) = 1$. We obtain the following formula.
\begin{align*}
&\qdim \Hom{}{\cyclicGroupAlg{p^{n}}}{\cyclotomicField{p^{n-s}}}{\cc^{r-n_{-}(D),\ast}(D)} = \\
&= \sum_{v = 0}^{\min(s, u_{r})} \sum_{s \in \mathcal{S}_{r}^{p^{v}}(D)} \jones_{p^{v},1}(s) + 
\sum_{v = \min(s,u_{r})+1}^{u_{r}} \sum_{s \in \mathcal{S}_{r}^{p^{v}}} \jones_{p^{v},p^{v-s}}(s).
\end{align*}
Plugging the above formula into (\ref{eqn: equiv Jones pol sum}) and taking the difference 
$$\jones_{p^{n},p^{n-s}}(D) - \jones_{p^{n},p^{n-s+1}}(D)$$
yields the desired formula.
\end{proof}

\bibliography{biblio}

\noindent W. Politarczyk: Department of Mathematics and Computer Science, Adam Mickiewicz 
University in Poznań,

\noindent e-mail: \textit{politarw@amu.edu.pl}

\end{document}